\documentclass[11pt,twoside]{article}
\usepackage{graphicx}
\usepackage[misc]{ifsym}
\usepackage{amsfonts}
\usepackage{amssymb,amsbsy}
\usepackage{amsmath,amssymb,amsthm,fancybox,array}
\usepackage{enumerate}
\usepackage[utf8]{inputenc}
\usepackage[T1]{fontenc}
\usepackage{pxfonts}
\usepackage{natbib}
\usepackage{dsfont}
\usepackage{caption}
\usepackage{subcaption}

\usepackage[top=50mm, bottom=20mm, left=30mm, right=40mm]{geometry}

\theoremstyle{plain}

\newtheorem{proposition}{Proposition}[section]

\theoremstyle{definition}
\newtheorem{definition}{Definition}[section]
\newtheorem{example}{Examples}[section]

\theoremstyle{remark}
\newtheorem{remark}{Remark}

\theoremstyle{corollary}
\newtheorem{corollary}{Corollary}[section]

\numberwithin{equation}{section}

\begin{document}
\setcounter{page}{1}
\thispagestyle{empty}
\vskip-0.5cm
\noindent {\scriptsize \bf JIRSS (Year) \\
Vol. xx, No. x, pp xx-xx\\
DOI:000000000000000000000000000000}

\markboth{~~\hrulefill~~ \small Buono et al.} 
{\small Dispersion indices based on uncertainty measures  ~~\hrulefill~~ }

\topmargin=0mm
\vspace{2cm}
{\Large \bf \noindent Dispersion indices based on Kerridge inaccuracy and Kullback-Leibler divergence}\\[0.7cm]
{\bf \noindent Francesco Buono $^1$, Camilla Calì $^2$, Maria Longobardi $^2$}\\ \\ 
{\small  $^1$  Dipartimento di Matematica e Applicazioni ``Renato Caccioppoli'', Università degli Studi di Napoli Federico II, Naples, Italy; \\
$^2$ Dipartimento di Biologia, Università degli Studi di Napoli Federico II, Naples, Italy. }
\renewcommand{\thefootnote}{}
\footnotetext{\hskip-0.6cm F. Buono (\Letter) (francesco.buono3@unina.it), C. Calì (camilla.cali@unina.it), M. Longobardi (maria.longobardi@unina.it).}

\vspace{0.7cm}
{\noindent \bf Abstract.}
The concept of varentropy has been recently introduced as a dispersion index of the reliability of measure of information. In this paper, we introduce new measures of variability for two measures of uncertainty, the Kerridge inaccuracy measure and the Kullback-Leibler divergence. These new definitions and related properties, bounds and examples are presented. Finally we show an application of Kullback-Leibler divergence and its dispersion index using the mean-variance rule. \\[2mm]

\renewcommand{\thefootnote}{\arabic{footnote}}
\setcounter{footnote}{0}

\vspace*{0.3cm}

\section{Introduction}

Let $X$ be a non negative and absolutely continuous random variable with cumulative distribution function (cdf) $F$ and probability density function (pdf) $f$. \cite{shannon} introduced a measure of uncertainty as the average level of information associated to the random variable $X$. This is known as Shannon entropy or differential entropy and is defined as
\begin{equation}
\nonumber
H(X)=\mathbb E_f[-\log f(X)]=-\int_0^{+\infty} f(x)\log f(x) dx,
\end{equation}
where $\log$ is the natural logarithm. Since then, several properties of Shannon entropy were studied and generalizations of this measure were introduced. We can observe that the Shannon entropy is position-free, in the sense that $X$ and $X+b$, with $b\in\mathbb R$, have the same entropy. 


The notion of differential entropy has been extended to study the discrepancy between two distributions. In the context of measures discussed in this paper, $f$ and $g$ are two pdf's associated with a single random variable $X$ in problems in which $f$ is the pdf of the ``true” distribution of
$X$ and $g$ is the pdf suggested by the results of an experiment (\cite{kerridge}). In the other case, $f$ is suitable to be selected since it is closest to a reference pdf
$g$ (\cite{kullback}). For clarity of presentation, let us consider two absolutely continuous non negative random variables $X$ and $Y$ with cdf's $F, G$ and pdf’s $f, g$, respectively. If $F$ is the distribution corresponding to the observations and $G$ is the distribution assigned by the experimenter, then the inaccuracy measure of $X$ and $Y$ (also named cross entropy of $Y$ on $X$ or relative distance
between $X$ and $Y$) is given by \cite{kerridge}
\begin{equation}
\label{eqki}
I(f;g)\equiv H_f(g)\coloneqq\mathbb E_f[-\log g(X)]=-\int_0^{+\infty} f(x) \log g(x) dx.
\end{equation}
As in the previous definition, the inaccuracy is an extension of the entropy $H(X)$.
This measure of uncertainty has been widely studied in the literature in order to adapt it to different contexts (see for instance \cite{ghosh, khorashadizadeh, kundu}). Moreover, \cite{taneja} introduced and studied the weighted version of inaccuracy that is a shift dependent measure of uncertainty, whereas the study of residual and past lifetime distributions through the inaccuracy is provided in \cite{tanejak} and \cite{kumar}, respectively.

As an information distance between two random variables $X$ and $Y$, \cite{kullback} proposed a directed divergence defined as
\begin{equation}
\label{eqkl}
K(f:g)=\mathbb E_f\left[\log \frac{f(X)}{g(X)}\right]=\int_0^{+\infty} f(x) \log \frac{f(x)}{g(x)} dx ,
\end{equation}
it is also known as information divergence, information gain, relative entropy or discrimination measure.  The Kullback-Leibler divergence is a measure of the similarity (closeness) between the two distributions and it plays an important role in information theory, reliability and other related fields. Several extensions of this measure have been proposed in the literature, for instance, one may refer to \cite{park, sunoj}. Moreover, we remark that the Kullback-Leibler divergence is non negative and equal to 0 if and only if $X$ and $Y$ are identically distributed. This characterization
property allows to use the estimated Kullback-Leibler information as a goodness of fit test statistic, see \cite{arizono, balakrishnan} for more details on this topic. Finally, the Kullback-Leibler divergence and the inaccuracy are related by the following relation
\begin{equation}
\label{rel}
K(f:g)=I(f;g)-H(X).
\end{equation}

Recently, the study of the variability of measures of information captured the interest of researchers. In fact, a dispersion index is useful to understand the reliability of the measure. In this perspective, \cite{fradelizi} study the concept of varentropy defined by
\begin{equation}
\label{varentropy}
VarH(X)\coloneqq Var_f[-\log f(X)]= \int_0^{+\infty} f(x)\log^2 f(x) dx- [H(X)]^2,
\end{equation}
whereas \cite{goodarzi} provide a useful bound of it. It is clear that the notation $VarH(X)$ is only a way to write the varentropy since it is not the variance of the entropy.

In this paper, we study the variability of the measures of uncertainty recalled above. In fact, we pointed out that those measures can be defined as expectations. Hence, we can evaluate their dispersion through the variance, in the sense that a measure with a lower level of variance can be assumed as more reliable. More precisely, the paper is organized as follows. In Section 2 and in Section 3 we introduce a dispersion index of Kerridge inaccuracy and Kullback-Leibler divergence, respectively. For these new definitions we provide properties, bounds and examples. In section 4 we use the mean-variance rule in order to apply the dispersion index of Kullback-Leibler divergence to some illustrative examples.

\section{Varinaccuracy}

The inaccuracy can be expressed in terms of the expectation of $-\log g(X)$ (see \eqref{eqki}) and for this reason it is useful to study the variance of this random variable. In the following definition, we introduce the varinaccuray as a dispersion index based on the inaccuracy, also known as cross entropy. 

\begin{definition}
\label{def1}
Let $X$ and $Y$ be two non negative random variables with pdf's $f$ and $g$, respectively. The varinaccuracy of $X$ and $Y$ can be defined as 
\begin{eqnarray}
\nonumber
VarI(f;g)&\coloneqq& Var_f [-\log g(X)] \\
\label{eqvarki}
&=&\int_0^{+\infty} f(x) \log^2 g(x) dx-\left[I(f;g)\right]^2.
\end{eqnarray}
\end{definition}

Also in this definition, $VarI(f;g)$ does not represent the variance of $I(f;g)$ but it is only a notation. Of course, if $X$ and $Y$ are identically distributed, then, as the inaccuracy reduces to Shannon entropy, the varinaccuracy reduces to the well-known varentropy \eqref{varentropy}.

\begin{remark}
We specify that Definition \ref{def1} can be given in a more general context omitting the non-negativity assumption. In this case, all integrals have to be understood as extended to the common support of $X$ and $Y$.
\end{remark}

Now, we give some examples of evaluation of varinaccuracy for different kinds of distributions.

\begin{example}
\label{ex1}
Consider $X\sim Exp(1)$ and $Y\sim Exp(2)$. Then by \eqref{eqki} the inaccuracy measure of $X$ and $Y$ is given by
\begin{equation}
\nonumber
I(f;g)=-\int_0^{+\infty} e^{-x} \log\left(2e^{-2x}\right) dx =2-\log2.
\end{equation}
Hence, the varinaccuracy is obtained by \eqref{eqvarki} as
\begin{eqnarray}
\nonumber
VarI(f;g)&=& \int_0^{+\infty} e^{-x} \log^2\left(2e^{-2x}\right) dx - (2-\log2)^2 \\
\nonumber
&=& \log^2 2 -4\log2+8-4-\log^2 2+4\log2=4.
\end{eqnarray}
In order to generalize the above example, we consider now $X\sim Exp(\lambda)$ and $Y\sim Exp(\eta)$. In Figure \ref{fig-exp}, the inaccuracy and the varinaccuracy of $X$ and $Y$ are plotted as functions of $\eta$, for $\lambda=1,2,3,4$, with solid, dashed, dotted and dash-dot line, respectively. Observe that $I(f;g)$ has minimum at $\eta=\lambda$ and $VarI(f;g)$ is increasing in $\eta$.

\begin{figure}
     \centering
     \begin{subfigure}[b]{0.48\textwidth}
         \centering
         \includegraphics[width=\textwidth]{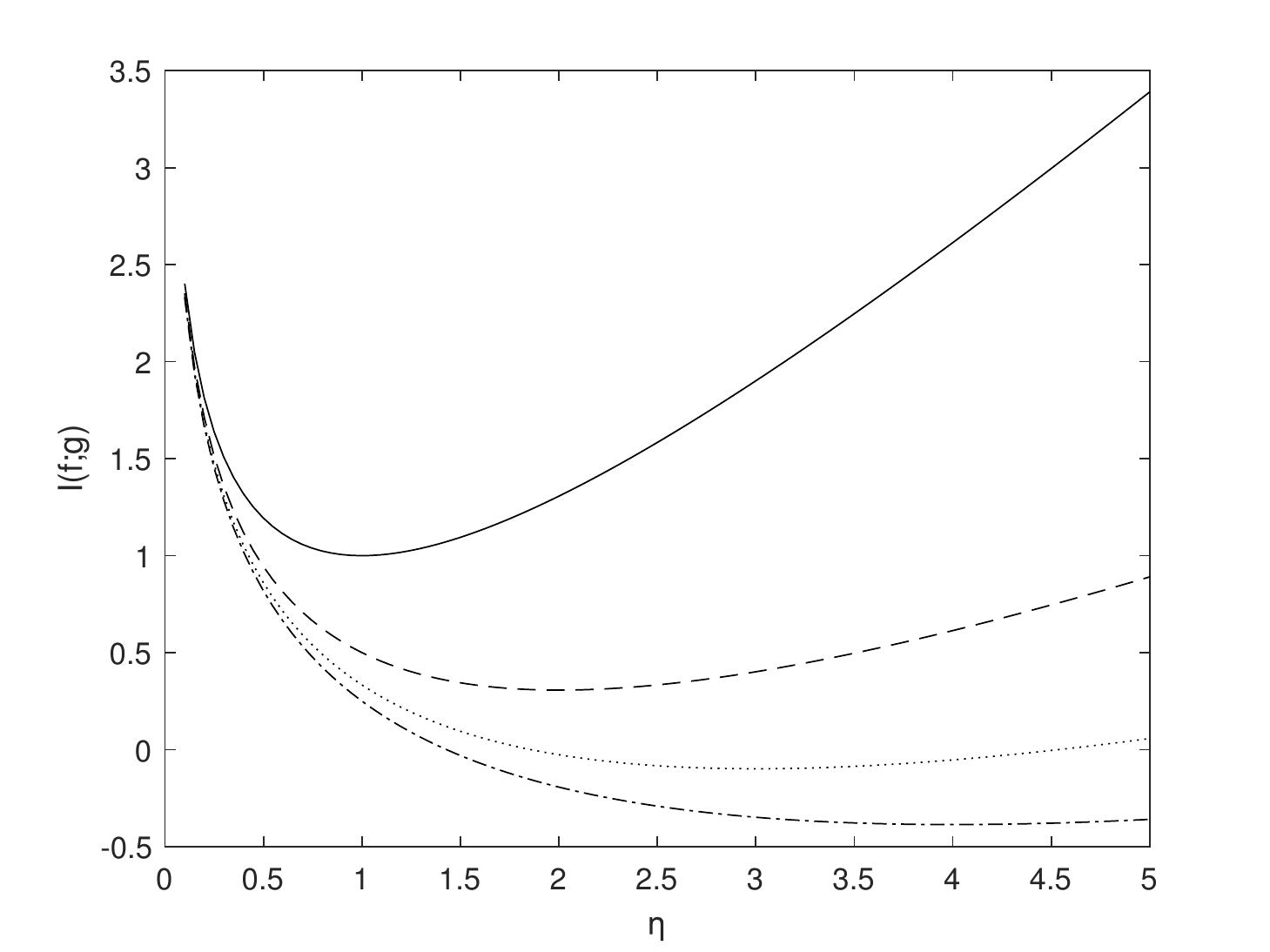}
         \caption{$I(f;g)$}
     \end{subfigure}
     \hfill
     \begin{subfigure}[b]{0.48\textwidth}
         \centering
         \includegraphics[width=\textwidth]{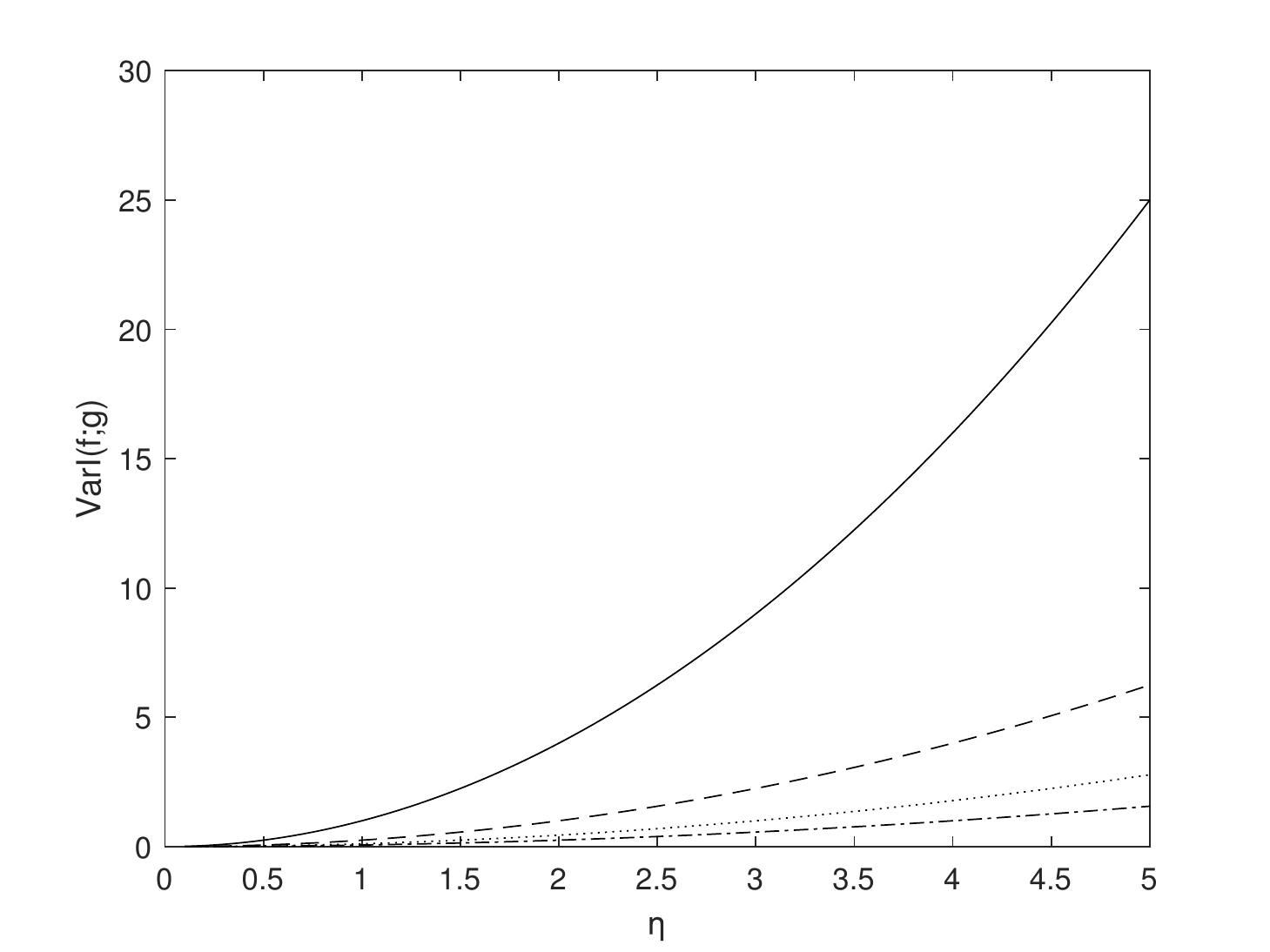}
         \caption{$VarI(f;g)$}
     \end{subfigure}
     \hfill
        \caption{Plot of $I$ and $VarI$ in Example \ref{ex1} as a function of $\eta$ for $\lambda=1,2,3,4$.}
        \label{fig-exp}
\end{figure}
\end{example}

\begin{example}
\label{ex3}
Consider $X\sim U(0,1)$ and $Y$ with probability density function $g$ given by
\begin{equation}
\nonumber
g(y)=2y, \ \ \ y\in(0,1).
\end{equation}
Then, by \eqref{eqki}, the inaccuracy measure of $X$ and $Y$ is given by
\begin{equation}
\nonumber
I(f;g)=-\int_0^{1}  \log\left(2x\right) dx = 1-\log 2.
\end{equation}
Hence, the varinaccuracy is obtained by \eqref{eqvarki} as
\begin{eqnarray}
\nonumber
VarI(f;g)&=& \int_0^{1} \log^2\left(2x\right) dx - \left(1-\log 2\right)^2 \\
\nonumber
&=& \log^2 2-2\log2+2-1-\log^2 2+2\log2 =1.
\end{eqnarray}
More in general, we can say that $Y$ belongs to the family of Power distributions characterized by the pdf $g_{\alpha}(x)=\alpha x^{\alpha-1}$, $x\in(0,1)$. In Figure \ref{fig-power}, we plot the inaccuracy and the varinaccuracy of $X$ and $Y$ as a function of $\alpha$. In this case, the inaccuracy reaches the minimum at $\alpha=1,$ that is when $Y$ has a uniform distribution in $(0,1)$, and $VarI$ is not monotone.

\begin{figure}
     \centering
     \begin{subfigure}[b]{0.48\textwidth}
         \centering
         \includegraphics[width=\textwidth]{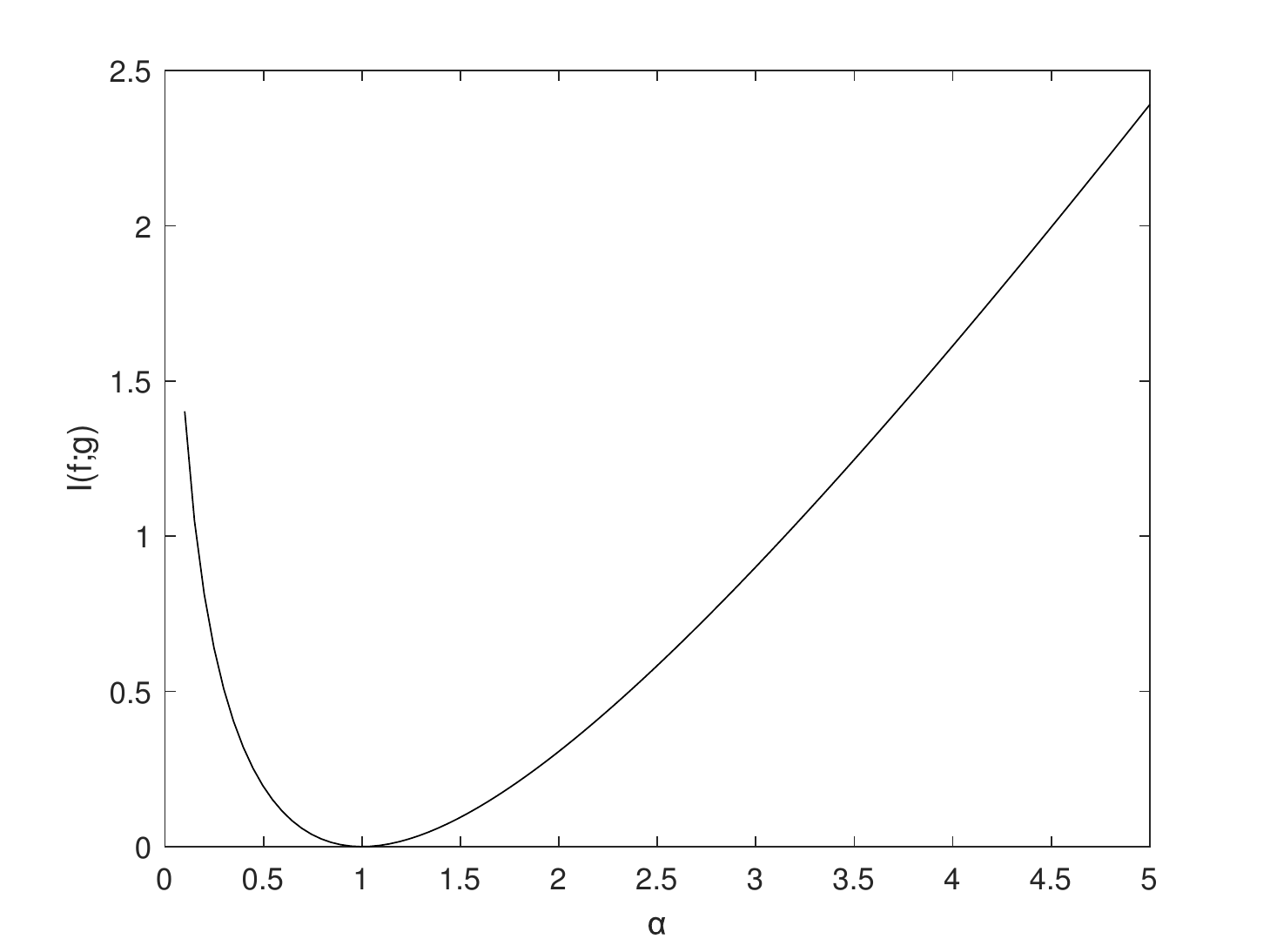}
         \caption{$I(f;g)$}
     \end{subfigure}
     \hfill
     \begin{subfigure}[b]{0.48\textwidth}
         \centering
         \includegraphics[width=\textwidth]{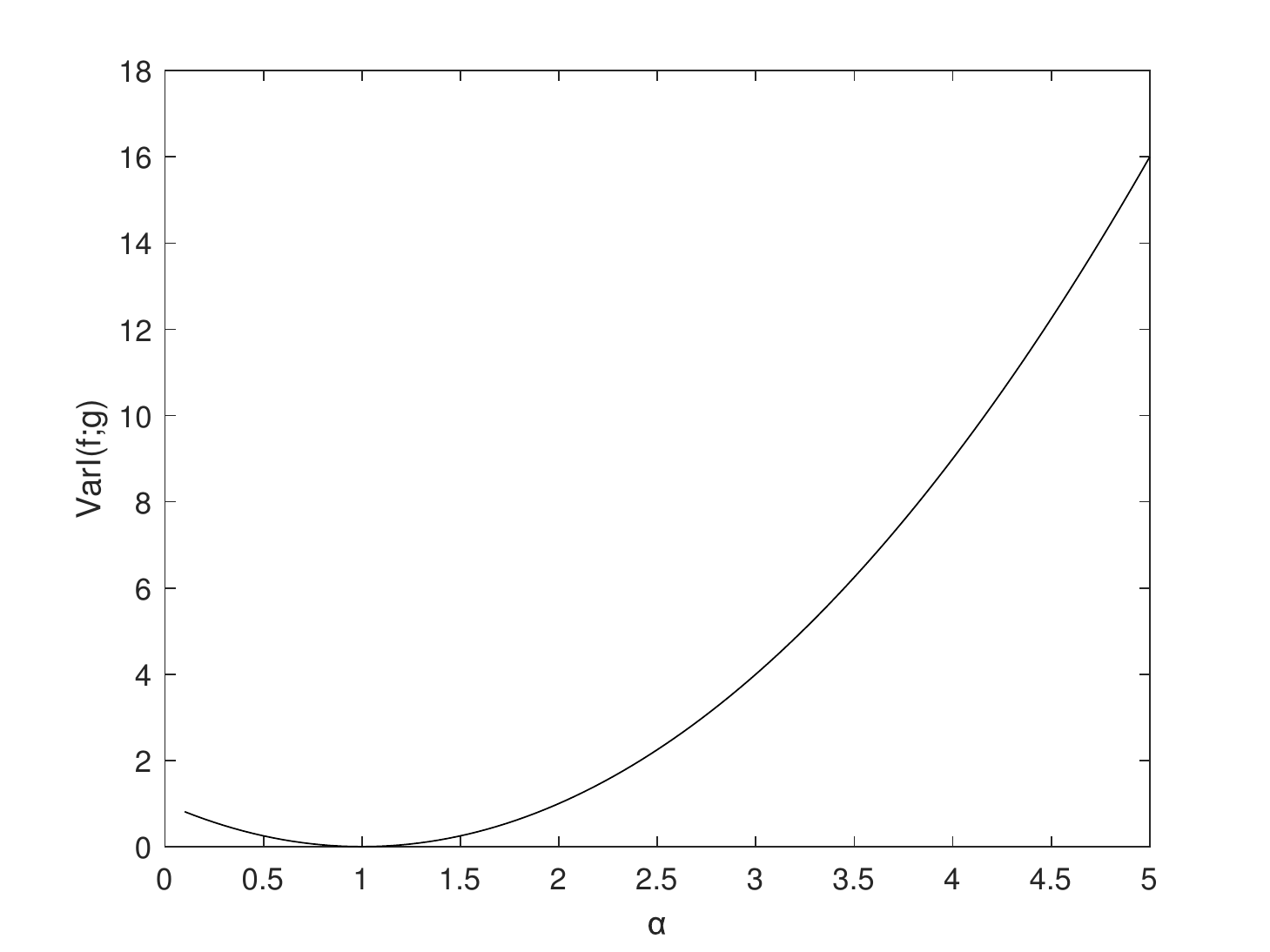}
         \caption{$VarI(f;g)$}
     \end{subfigure}
     \hfill
        \caption{Plot of $I$ and $VarI$ in Example \ref{ex3} as a function of $\alpha$.}
        \label{fig-power}
\end{figure}
\end{example}

In the following proposition, we examine the behaviour of varinaccuracy under affine transformations. The proof of it can be obtained by a simple change of variable technique and it is similar to the property of Shannon entropy along with invariance of the variance under translation.

\begin{proposition}
Let $X$ and $Y$ be two random variables with common support $S$ and pdf's $f$ and $g$, respectively. Let $a>0$, $b\geq0$ and define the variables $\tilde X$, $\tilde Y$ as $\tilde X=aX+b$, $\tilde Y=aY+b$ with pdf's $\tilde f$ and $\tilde g$, respectively. Then, we have
\begin{equation}
\nonumber
VarI(\tilde f;\tilde g)=VarI(f;g).
\end{equation}
\end{proposition}


\begin{proposition}
Let $X$ and $Y$ be two random variables with common support $S$ and pdf's $f$ and $g$, respectively. Let $\phi$ be a strictly monotone function and define the variables $\tilde X$, $\tilde Y$ as $\tilde X=\phi(X)$, $\tilde Y=\phi(Y)$ with pdf's $\tilde f$ and $\tilde g$, respectively. Then, we have
\begin{equation}
\nonumber
VarI(\tilde f;\tilde g)=VarI(f;g)+ Var_f[\log |\phi'(X)|]-2cov_f(\log g(X),\log|\phi'(X)|).
\end{equation}
\end{proposition}

\begin{proof}
Without loss of generality, consider the case $S=(0,+\infty)$ with $\phi$ strictly increasing from $\phi(0)$ to $+\infty$. Then the common support of $\tilde X$ and $\tilde Y$ is $(\phi(0),+\infty)$. The relation between the pdf's of $\tilde X$, $\tilde Y$ and $X$, $Y$ is given by
\begin{equation}
\nonumber
\tilde f(x)= \frac{f(\phi^{-1}(x))}{\phi'(\phi^{-1}(x))}, \ \ \  \tilde g(x)=\frac{g(\phi^{-1}(x))}{\phi'(\phi^{-1}(x))}, \ \ \ x\in(\phi(0),+\infty),
\end{equation}
where $\frac{1}{\phi'(\phi^{-1}(x))}=\left.\frac{d}{dy} \phi^{-1}(y)\right |_{y=\phi^{-1}(x)}$.
Then, the inaccuracy of $\tilde X$ and $\tilde Y$ can be written as
\begin{eqnarray}
\nonumber
I(\tilde f;\tilde g)&=& - \int_{\phi(0)}^{+\infty} \frac{f(\phi^{-1}(x))}{\phi'(\phi^{-1}(x))} \log \left[\frac{g(\phi^{-1}(x))}{\phi'(\phi^{-1}(x))}\right] dx \\
\nonumber
&=&I(f;g)+\mathbb E_f[\log \phi'(X)].
\end{eqnarray}
Hence, the varinaccuracy of $\tilde X$ and $\tilde Y$ can be obtained as
\begin{eqnarray}
\nonumber
VarI(\tilde f;\tilde g)&=& \int_b^{+\infty} \frac{f(\phi^{-1}(x))}{\phi'(\phi^{-1}(x))} \log^2 \left[\frac{g(\phi^{-1}(x))}{\phi'(\phi^{-1}(x))}\right] dx -[I(f;g)+\mathbb E_f[\log \phi'(X)]]^2\\
\nonumber
&=& \int_0^{+\infty} f(x)\log^2 g(x) dx +\int_0^{+\infty}f(x)\log^2(\phi'(x)) dx \\
\nonumber
&&-2\int_0^{+\infty} f(x) \log (\phi'(x))\log (g(x)) dx -[I(f;g)+\mathbb E_f[\log \phi'(X)]]^2\\
\nonumber
&=& VarI(f;g)+ Var_f[\log \phi'(X)]-2cov_f(\log g(X),\log\phi'(X)),
\end{eqnarray}
which completes the proof.
\end{proof}

\begin{proposition}
Let $X$ and $Y$ be two random variables with common support $S$ and pdf's $f$ and $g$, respectively. Then, $VarI(f;g)=0$ if and only if $Y$ is uniformly distributed in $S$.
\end{proposition}

\begin{proof}
The varinaccuracy is defined as a variance, it vanishes only for degenerate distributions. In particular, $\log g(x)$ needs to be constant for $x\in S$, i.e., $g$ needs to be a constant function and then $Y$ has to be uniformly distributed in $S$.
\end{proof}

In the following proposition, we obtain a lower bound for the varinaccuracy based on Chebyshev inequality, which for a random variable $W$ with mean $\mathbb E(W)$ and variance $Var(W)$ is given by
\begin{equation}
\label{ceb}
\mathbb P\left(|W-\mathbb E(W)|<\varepsilon\right)\geq 1- \frac{Var(W)}{\varepsilon^2},  \ \ \ \varepsilon>0.
\end{equation}

\begin{proposition}
\label{propceb}
Let $X$ and $Y$ be two random variables with common support $S$ and pdf's $f$ and $g$, respectively and let $\varepsilon>0$. Then a lower bound for the varinaccuracy is given by
\begin{equation}
VarI(f;g)\geq \varepsilon^2\left[ \mathbb P\left(g(X)\leq e^{-\varepsilon-I(f;g)}\right)+\mathbb P\left(g(X)\geq e^{\varepsilon-I(f;g)}\right)\right].
\end{equation}
\end{proposition}

\begin{proof}
Based on the definitions of inaccuracy and varinaccuracy \eqref{eqki}, \eqref{eqvarki}, the Chebyshev inequality \eqref{ceb} yields
\begin{equation}
\label{eqpro1}
VarI(f;g)\geq\varepsilon^2\mathbb P(|\log g(X)+I(f;g)|\geq \varepsilon).
\end{equation}
The second factor in the right hand side of the above equation can be written as
\begin{eqnarray}
\nonumber
\mathbb P(|\log g(X)+I(f;g)|\geq \varepsilon)&=&\mathbb P\left(\log g(X)+I(f;g)\leq -\varepsilon\right)+\mathbb P\left(\log g(X)+I(f;g)\geq \varepsilon\right) \\
\label{eqpro2}
&=& \mathbb P\left(g(X)\leq e^{-\varepsilon-I(f;g)}\right)+\mathbb P\left(g(X)\geq e^{\varepsilon-I(f;g)}\right),
\end{eqnarray}
and the proof is completed by combining \eqref{eqpro1} and \eqref{eqpro2}.
\end{proof}

In the following corollaries, we specialize the result of Proposition \ref{propceb} when $g$ is strictly increasing or decreasing.

\begin{corollary}
\label{cor1}
Let $X$ and $Y$ be two random variables with common support $S$, pdf's $f$ and $g$ and cdf's $F$ and $G$, respectively and let $\varepsilon>0$. If $g$ is strictly decreasing in $S$, then
\begin{equation}
VarI(f;g)\geq \varepsilon^2\left[ \overline F\left(g^{-1}(e^{-\varepsilon-I(f;g)})\right)+ F\left(g^{-1}(e^{\varepsilon-I(f;g)})\right)\right],
\end{equation}
where $\overline F(\cdot)=1-F(\cdot)$ is the survival function of $X$.
\end{corollary}

\begin{corollary}
\label{cor2}
Let $X$ and $Y$ be two random variables with common bounded support $S$, pdf's $f$ and $g$ and cdf's $F$ and $G$, respectively, and let $\varepsilon>0$. If $g$ is strictly increasing in $S$, then
\begin{equation}
VarI(f;g)\geq  \varepsilon^2\left[F\left(g^{-1}(e^{-\varepsilon-I(f;g)})\right)+ \overline F\left(g^{-1}(e^{\varepsilon-I(f;g)})\right)\right].
\end{equation}
\end{corollary}

\begin{example}
\label{exbound1}
Consider $X\sim Exp(\lambda)$ and $Y\sim Exp(\eta)$. In Example \ref{ex1}, we have plotted the varinaccuracy of $X$ and $Y$. Here, we use Corollary \ref{cor1} to evaluate the lower bound. In fact, in this case the pdf $g$ of $Y$ is strictly decreasing and we have
\begin{equation}
\nonumber
g^{-1}(z)=-\frac{1}{\eta} \log \frac{z}{\eta}, \ \ \ z\in(0,\eta).
\end{equation}
Moreover, the inaccuracy of $X$ and $Y$ is given by
\begin{equation}
\nonumber
I(f;g)=-\log\eta +\frac{\eta}{\lambda}.
\end{equation}
If $\varepsilon\lambda>\eta$, we have
\begin{equation}
\nonumber
e^{\varepsilon-I(f;g)}>\eta,
\end{equation}
and then $\mathbb P(g(X)\geq e^{\varepsilon-I(f;g)})=0$. Thus, we can conclude
\begin{equation}
VarI(f;g)\geq
\begin{cases}
\varepsilon^2\left(e^{-1-\varepsilon\lambda/\eta}+1-e^{-1+\varepsilon\lambda/\eta}\right), &\mbox{ if } \varepsilon\lambda\leq\eta\\
\varepsilon^2e^{-1-\varepsilon\lambda/\eta}, & \mbox{ if } \varepsilon\lambda>\eta.
\end{cases}
\end{equation}
In Figure \ref{fig-expbound} we plot the varinaccuracy and the bound in the case $\lambda=4$ as a function of $\eta$ and with different choices of $\varepsilon$.
\begin{figure}[ht]
\centering
\includegraphics[scale=0.5]{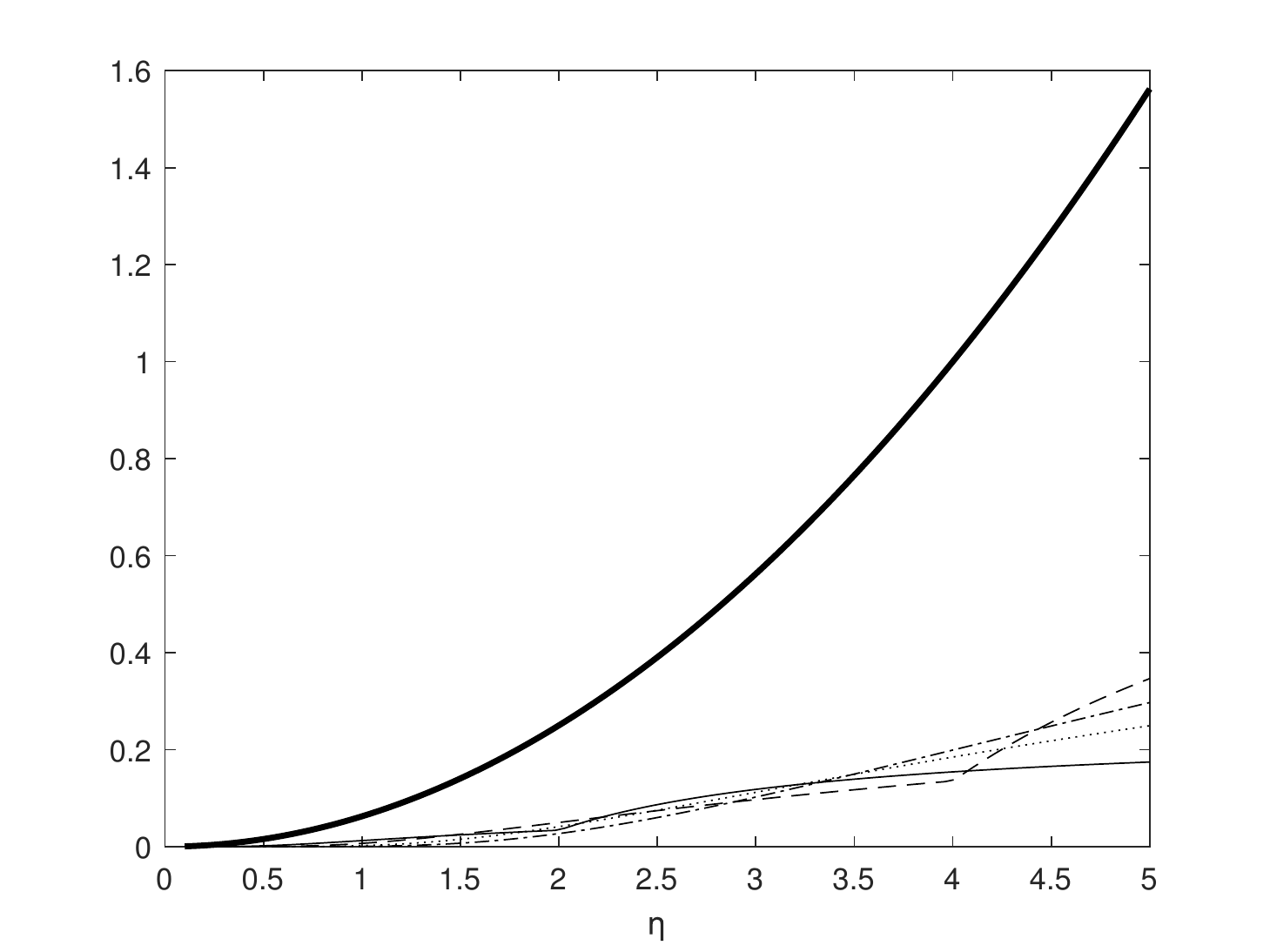}
\caption{Plot of $VarI(f;g)$ (thick line) and lower bounds in Example \ref{exbound1} as a function of $\eta$ for $\lambda=4$ and $\varepsilon=0.5,1,1.5,2$ (solid, dashed, dotted and dash-dot line, respectively).}
\label{fig-expbound}
\end{figure}
\end{example}

\begin{example}
\label{exbound2}
Consider $X\sim U(0,1)$ and $Y\sim Power(\alpha)$ with $\alpha>1$. In Example \ref{ex3}, we have plotted the varinaccuracy of $X$ and $Y$. Here, Corollary \ref{cor2} is used to evaluate the lower bound. In fact, in this case the pdf $g$ of $Y$ is strictly increasing, and we have
\begin{equation}
\nonumber
g^{-1}(z)=\left(\frac{z}{\alpha}\right)^{\frac{1}{\alpha-1}}, \ \ \ z\in(0,\alpha).
\end{equation}
Moreover, the inaccuracy of $X$ and $Y$ is given by
\begin{equation}
\nonumber
I(f;g)=-\log\alpha+(\alpha-1).
\end{equation}
If $1<\alpha<1+\varepsilon$, we have
\begin{equation}
\nonumber
e^{\varepsilon-I(f;g)}>\alpha,
\end{equation}
and then $\mathbb P(g(X)\geq e^{\varepsilon-I(f;g)})=0$. Thus, we can conclude
\begin{equation}
VarI(f;g)\geq
\begin{cases}
\varepsilon^2\left(e^{\frac{1-\varepsilon-\alpha}{\alpha-1}}+1-e^{\frac{1+\varepsilon-\alpha}{\alpha-1}}\right), &\mbox{ if } \alpha\geq1+\varepsilon\\
\varepsilon^2e^{\frac{1-\varepsilon-\alpha}{\alpha-1}}, & \mbox{ if } 1<\alpha<1+\varepsilon.
\end{cases}
\end{equation}
In Figure \ref{fig-powerbound} we plot the varinaccuracy and the bound as a function of $\alpha$ with different choices of $\varepsilon$.
\begin{figure}[ht]
\centering
\includegraphics[scale=0.5]{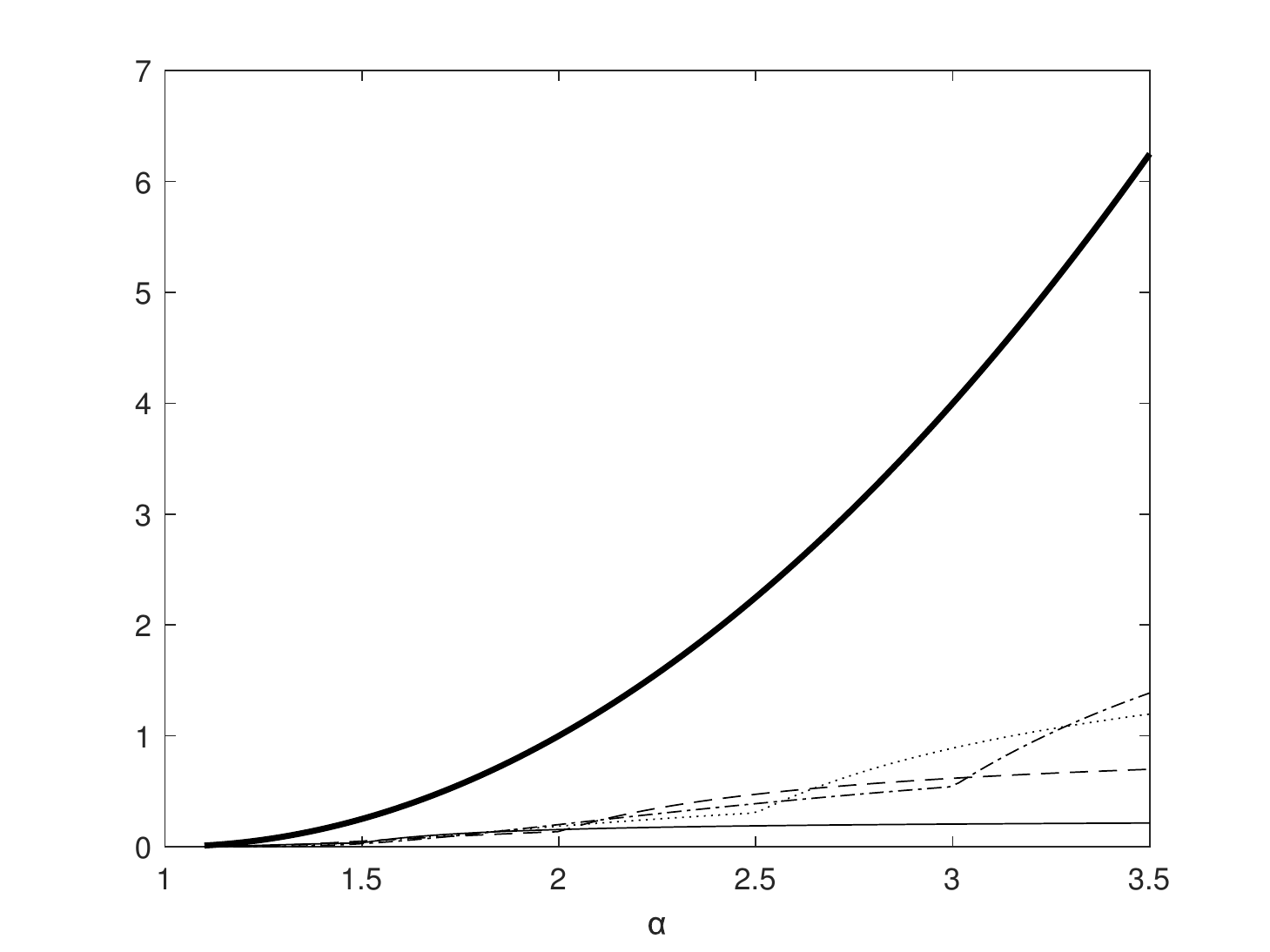}
\caption{Plot of $VarI(f;g)$ (thick line) and lower bounds in Example \ref{exbound2} as a function of $\alpha$ for $\varepsilon=0.5,1,1.5,2$ (solid, dashed, dotted and dash-dot line, respectively).}
\label{fig-powerbound}
\end{figure}
\end{example}

\section{A dispersion index of Kullback-Leibler divergence}
In the following definition, we introduce a dispersion index of Kullback-Leibler divergence based on \eqref{eqkl}.
\begin{definition}
\label{def2}
Let $X$ and $Y$ be two non-negative random variables with pdf's $f$ and $g$, respectively. A dispersion index of Kullback-Leibler divergence of $X$ and $Y$ can be defined as
\begin{eqnarray}
\nonumber
VarK(f:g)&\coloneqq& Var_f \left[\log \frac{f(X)}{g(X)}\right] \\
\nonumber
&=& \mathbb E_f\left[\log^2 \frac{f(X)}{g(X)}\right] - \left[K(f:g)\right]^2\\
\label{varkl}
&=&\int_0^{+\infty} f(x) \log^2 \frac{f(x)}{g(x)} dx-\left[K(f:g)\right]^2 
\end{eqnarray}
\end{definition}
We point out that $VarK(f:g)$ does not represent the variance of Kullback-Leibler divergence but it is only a short notation.

\begin{remark}
As said for the varinaccuracy, also the definition of $VarK$ can be given for variables with a common support $S$ not necessarily equal to $(0,+\infty)$.
\end{remark}

In the following proposition, in analogy with the relation \eqref{rel} we study a connection between varentropy, varinaccuracy and $VarK$.

\begin{proposition}
Let $X$ and $Y$ be two non negative random variables with common support $S$ and pdf's $f$ and $g$, respectively. Then
\begin{equation}
\label{rel2}
VarK(f:g)=VarH(X)+VarI(f;g)-2cov_f(\log f(X),\log g(X)).
\end{equation}
\end{proposition}

\begin{proof}
By \eqref{varkl} and by taking into account the expression of the variance of the sum, we obtain
\begin{eqnarray}
\nonumber
VarK(f:g)&=&Var_f \left[\log \frac{f(X)}{g(X)}\right]=Var_f\left[\log f(X)-\log g(X) \right] \\
\nonumber
&=& Var_f[\log f(X)]+ Var_f[\log g(X)]-2cov_f(\log f(X),\log g(X)) \\
\nonumber
&=& Var_f[-\log f(X)]+ Var_f[-\log g(X)]-2cov_f(\log f(X),\log g(X))
\end{eqnarray}
and, by recalling \eqref{varentropy} and \eqref{eqvarki}, we get the thesis.
\end{proof}

\begin{proposition}
\label{propdiv}
Let $X$ and $Y$ be two random variables with common support $S$ and pdf's $f$ and $g$, respectively. Then, $VarK(f:g)=0$ if and only if $X$ and $Y$ are identically distributed.
\end{proposition}

\begin{proof}
$VarK$ is defined as a variance, hence it vanishes only for degenerate distributions. In particular, $\log \frac{f(x)}{g(x)}$ need to be constant for $x\in S$, i.e.,
\begin{equation}
\nonumber
\frac{f(x)}{g(x)}= a, \ \ \ x\in S,
\end{equation} 
where $a$ is a non-negative constant. In view of the normalization condition, we have $a=1$ and then $X$ and $Y$ are identically distributed.
\end{proof}

\begin{remark}
Proposition \ref{propdiv} enables to consider $VarK$ as a measure of divergence since it shares the positive-definiteness property with the Kullback-Leibler divergence. Moreover, as the Kullback-Leibler divergence, it can not be considered as a metric since it is not symmetric and does not satisfy the triangle inequality. The former is quite intuitive from the definition whereas the latter is shown by the following counterexample. Let $X$, $Y$ and $Z$ follow the Power distribution with parameters $0.5$, $3$ and $2$ and let us denote the pdf's with $f$, $g$, and $h$, respectively. An easy computation gives
\begin{equation}
\nonumber
VarK(f:g)=25, \ \  VarK(f:h)=9, \ \ VarK(h:g)=0.25,
\end{equation}
so that
\begin{equation}
\nonumber
VarK(f:g)> VarK(f:h)+VarK(h:g)
\end{equation}
and hence the triangle inequality is not satisfied.
\end{remark}

For furher developments, it could be possible to analyze the relationships among this new divergence measure and well-known measures as Kullback-Leibler, Rényi, Cressie-Read and Chernoff $\alpha$ divergences (see \cite{bedbur} for their definitions).

\section{VarK applications in detecting the underlying distribution}
The Kullback-Leibler divergence is a measure of similarity between two distributions. If we consider $X$ distributed as the data, then we can choose $Y$ in different ways in order to compare the values of $K(f:g)$, where $f$ and $g$ are the pdf's of $X$ and $Y$, respectively. Of course, a lower value of Kullback-Leibler divergence corresponds to an higher similarity of the distributions of $Y$ and data. There may be situations in which $Y_1$ and $Y_2$ are two different random variables with pdf's $g_1$ and $g_2$, respectively, and such that $K(f:g_1)\simeq K(f:g_2)$. In this case we can choose the more suitable distribution by considering $VarK$, in the sense that we could prefer a distribution with a lower variance even if it has an higher value of $K$.

In order to obtain a criterion based on Kullback-Leibler divergence and the related dispersion index, we set a threshold $r$ such that if $K(f:g_i)$, $i=1,2$ exceeds the value $r$ we can not accept such a distribution. To fix ideas, let us suppose $K(f:g_1)\leq K(f:g_2)$. Moreover, consider the case in which $K(f:g_2)<r$. As $K(f:g_2)$ tends to $r$ it becomes more difficult to prefer $Y_2$ to $Y_1$, but we can tolerate an higher value of the Kullback-Leibler divergence if we balance with a lower value of variance. Then, we use $VarK$ in order to standardize the difference between $r$ and $K$ and make comparisons. We prefer $Y_2$ to $Y_1$ if the following inequality is satisfied
\begin{equation}
\label{crit0}
\frac{r-K(f:g_1)}{\sqrt{VarK(f:g_1)}}<\frac{r-K(f:g_2)}{\sqrt{VarK(f:g_2)}}.
\end{equation}

\begin{remark}
Observe that the criterion in \eqref{crit0} is reasonable since when $K(f:g_1)=K(f:g_2)$, the variable with lower $VarK$ is preferred. Moreover, with the same variance, the variable with lower Kullback-Leibler divergence is still preferred. Finally, if $Y_1$ has lower both $K$ and $VarK$, it will be preferred to $Y_2$.
\end{remark}

In order to apply the criterion to concrete situations, we have to choose a value for the threshold $r$. It could be not convenient to fix a numerical value for $r$ but we can relate this quantity to the Kullback-Leibler divergences. In particular, we choose $r=2K(f:g_1)$, where $K(f:g_1)\leq K(f:g_2)$. Hence, the criterion in \eqref{crit0} can be reformulated in the following way: $Y_2$ is preferred to $Y_1$ if the following inequality is satisfied
\begin{equation}
\nonumber
\frac{K(f:g_1)}{\sqrt{VarK(f:g_1)}}<\frac{2K(f:g_1)-K(f:g_2)}{\sqrt{VarK(f:g_2)}}
\end{equation}
which is equivalent to
\begin{equation}
\label{crit}
K(f:g_2)<\left(2-\sqrt{\frac{VarK(f:g_2)}{VarK(f:g_1)}}\right)K(f:g_1).
\end{equation}

\begin{remark}
\label{discreteKL}
The same dispersion index given in Definiton \ref{def2} can be introduced also in the discrete case. When we have two discrete probability distributions $P$ and $Q$ defined on the same probability space $\mathcal {X}$, the Kullback-Leibler divergence of $P$ and $Q$ is defined as:
\begin{equation}
\label{KLdis}
K(P:Q)=\sum_{x \in \mathcal {X}} P(x) \log \frac{P(x)}{Q(x)}.
\end{equation}
The corresponding index of dispersion is
\begin{equation}
\label{Var-KLdis}
VarK(P:Q)=\sum_{x \in \mathcal {X}} P(x) \log^2 \frac{P(x)}{Q(x)}-[K(P:Q)]^2.
\end{equation}
\end{remark}

In the following, we give three applications of the above method in different scenarios. In the first one, we will have Kullback-Leibler divergences which do not satisfies the similarity property. In the second one, we will present the case in which we have two equal Kullback-Leibler divergences. In the third one, we will present the more critical situation, i.e., we will find a distribution with lower $K$ but with higher $VarK$.

\begin{example}
\label{exmon}
Consider in Table \ref{tab1} the data obtained from 200 repetitions of the experiment consisting in tossing 3 times a coin and recording how many times we get head.
\begin{table}[h]
\caption{Data of Example \ref{exmon}.}
\centering
\begin{tabular}{ccccc}
\hline
\textbf{Number of heads} 	 & 0	& 1 & 2 & 3 \\
\hline
\textbf{Number of observations} & 20 & 63 & 84 &  33 \\
\hline
\label{tab1}
\end{tabular}
\end{table}

If we denote by $X$ the random variable distributed as the data, from Table \ref{tab1} we get the distribution of $X$ as
\begin{equation*}
p_0=\mathbb P(X=0)=0.1, \ \ p_1=0.315, \ \ p_2=0.42, \ \ p_3=0.165.
\end{equation*}
Our intention is to establish a suitable distribution for the data, so we evaluate Kullback-Leibler divergence and its variance between $X$ and three different distributions $Y_1,Y_2,Y_3$, with probability mass functions $P,Q_1,Q_2,Q_3$, respectively. In particular, $Y_1$ follows a binomial distribution $B(3,0.55)$, where $0.55$ is obtained by maximum likelihood estimation, $Y_2$ follows a beta-binomial distribution with parameters $n=3$, $\alpha=12$ and $\beta=10$, and $Y_3$ follows a discrete uniform distribution over four elements. The values of Kullback-Leibler divergence and its variance are presented in Table \ref{tab2}.
\begin{table}[h]
\caption{$K(P:Q_i)$ and $VarK(P:Q_i)$, $i=1,2,3$, in Example \ref{exmon}.}
\centering
\begin{tabular}{lcc}
\hline
\textbf{Distribution} 	 & $K(P:Q)$	& $VarK(P:Q)$ \\
\hline
Binomial& 0.0011 & 0.0023\\
Beta-binomial & 0.0027 & 0.0054 \\
Uniform & 0.1305 & 0.2253 \\

\hline
\label{tab2}
\end{tabular}
\end{table}
Since the binomial distribution has lower Kullback-Leibler divergence and lower $VarK$, we can conclude that the binomial distribution is more appropriate than the Beta-binomial and the discrete uniform ones. Along the same lines, the Beta-binomial is preferred to the discrete uniform. 
\end{example}

\begin{example}
\label{ess}
Consider the real data (see Data Set 4.1 \cite{murthy}) which concern times till failures for 20 units: $11.24, 1.92, 12.74, 22.48, 9.60, 11.50, 8.86, 7.75, 5.73$, $9.37$, $30.42, 9.17, 10.20,$  $5.52, 5.85, 38.14, 2.99, 16.58, 18.92, 13.36$. The data are distributed as the random variable $X$ whose pdf is $f$. We estimate the density function through a kernel estimator with {\it MATLAB} function \texttt{ksdensity}. In order to establish if the distribution of the data is similar to a Weibull distribution $W2(\alpha,\lambda)$ with pdf 
\begin{equation*}
g(x)=\lambda \alpha x^{\alpha-1} \exp\left(-\lambda x^{\alpha}\right), \ \ x>0,
\end{equation*} 
we consider two different Weibull distribution, $Y_1\sim W2(1.5487,0.0166)$, with parameters given by maximum likelihood method, and $Y_2\sim W2(1.6,0.0127)$. In Figure \ref{weib} we present the plot of the estimated pdf of data and pdf's $g_1,g_2$ of $Y_1,Y_2$. 
\begin{figure}[ht]
\centering
\includegraphics[scale=0.7]{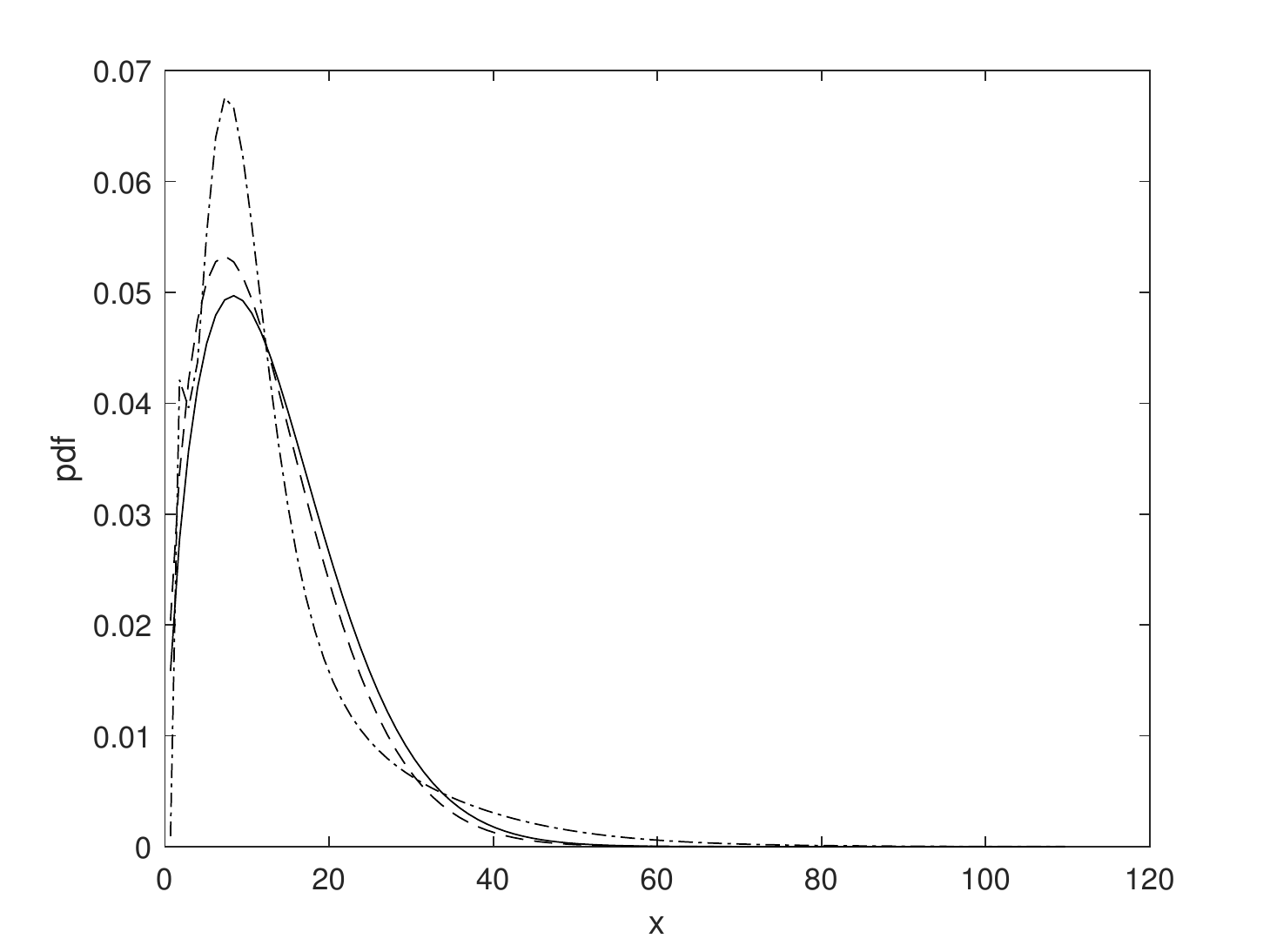}
\caption{Plot of pdf's of $X,Y_1,Y_2$ in Example \ref{ess} (dash-dot, dashed and solid line, respectively).}
\label{weib}
\end{figure}
With these distributions we obtain 
\begin{equation*}
K(f:g_1)=K(f:g_2)=0.0990.
\end{equation*}
Hence, in order to choose the more suitable distribution we have to compare the values of $VarK$ and we obtain
\begin{equation}
\nonumber
VarK(f:g_1)=0.3350 \ > \ VarK(f:g_2)=0.2936,
\end{equation}
and then we choose $Y_2$ since its Kullback-Leibler divergence has a lower variability.
\end{example}

\begin{example}
\label{exfin}
Consider the crab dataset given in \cite{murphy}. We focus on the distribution of the width of female crabs, represented by the random variable $X$ with pdf $f$, hence we have a sample of 100 units. We estimate the density function through a kernel estimator with {\it MATLAB} function \texttt{ksdensity}. We intend here to compare the distribution of the data with Weibull and Log-normal distributions. We recall that if $Y_2\sim Lognormal(\mu,\sigma)$, then the pdf is given by
\begin{equation}
\nonumber
g_2(x)=\frac{1}{x\sigma\sqrt{2\pi}}\exp\left(-\frac{(\log x-\mu)^2}{2\sigma^2}\right), \ \ x>0.
\end{equation}
In particular, by using the maximum likelihood estimation, we choose $Y_1\sim W2(5.6162,$ $1.1953e-09)$ and $Y_2\sim Lognormal(3.5559,0.2192)$.  In Figure \ref{figfin} we present the plot of the estimated pdf of data and pdf's of $Y_1,Y_2$. 
\begin{figure}[ht]
\centering
\includegraphics[scale=0.7]{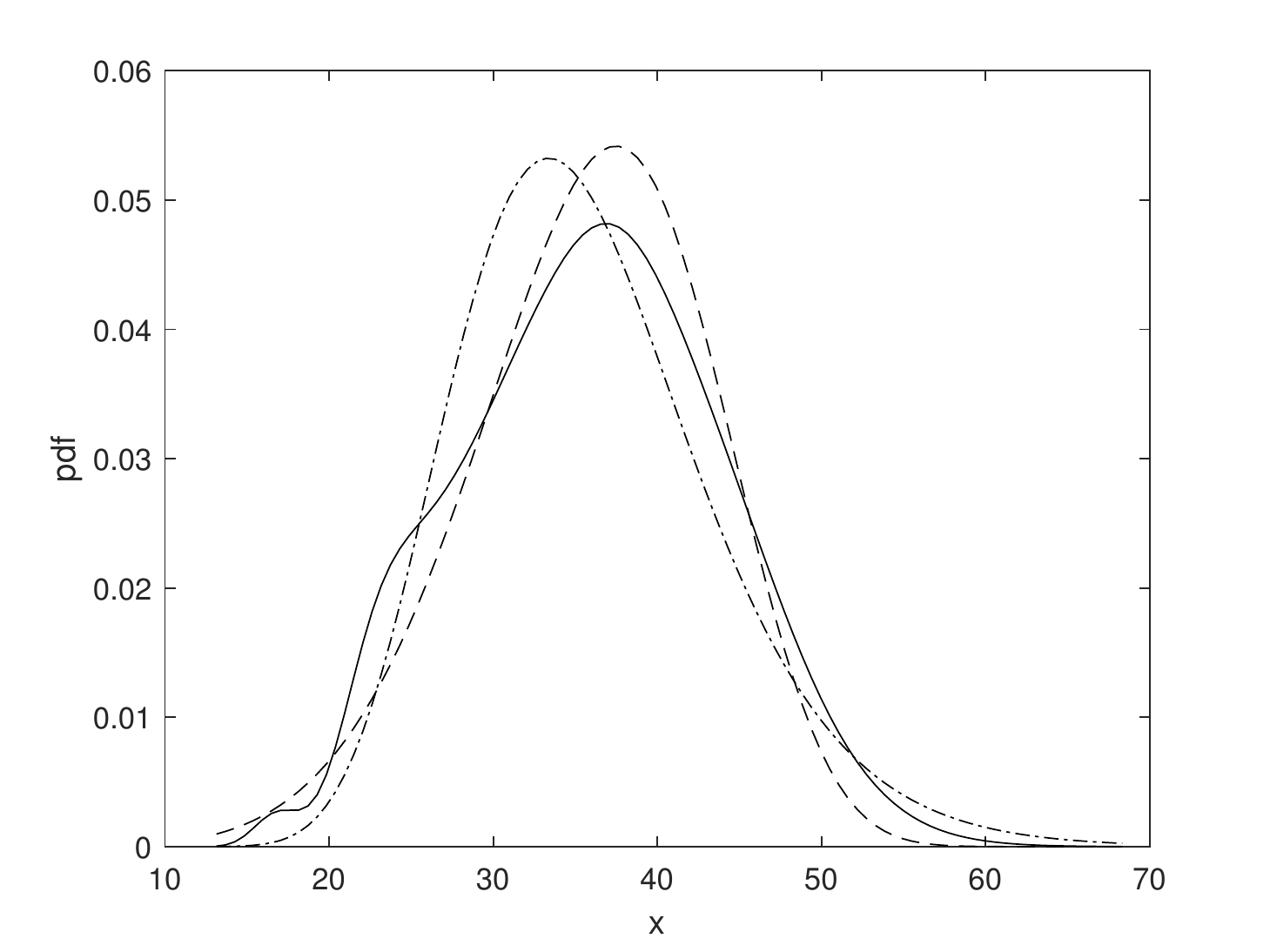}
\caption{Plot of pdf's of $X,Y_1,Y_2$ in Example \ref{exfin} (solid, dashed and dash-dot, respectively).}
\label{figfin}
\end{figure}
With these distributions we obtain 
\begin{align*}
&K(f:g_1)=0.0381, \ \ \ \ \  \ \ K(f:g_2)=0.0420, \\
&VarK(f:g_1)=0.1148, \ \ VarK(f:g_2)=0.0924.
\end{align*}
Hence, we are in the case in which $Y_1$ has lower Kullback-Leibler divergence but higher $VarK$ and the difference between $K(f:g_2)$ and $K(f:g_1)$ is small enough. Then, in order to choose the most suitable distribution, we use the criterion given in \eqref{crit} and compute the difference
\begin{equation}
\nonumber
K(f:g_2)-\left(2-\sqrt{\frac{VarK(f:g_2)}{VarK(f:g_1)}}\right)K(f:g_1)=-3.8085e-05.
\end{equation}
Thus the inequality in \eqref{crit} is satisfied and we can choose $Y_2$ as the distribution that fits the data in the best way.
\end{example}

\section{Conclusion}
In this paper we have introduced new measures of variability for some measures of uncertainty, in particular for the Kerridge inaccuracy measure and the Kullback-Leibler divergence. We have defined a dispersion index based on the Kerridge inaccuracy, $VarI$, named varinaccuracy. We have discussed the effect of linear transformations and strictly monotone functions on varinaccuracy and then lower bounds have been presented. A dispersion index of Kullback-Leibler divergence, $VarK$, and a connection among varentropy, varinaccuracy and $VarK$ have been introduced. Since the Kullback-Leibler divergence is a measure of similarity between two distributions, $VarK$ has been used to compare two distributions chosen to fit the data. In order to obtain a criterion based on Kullback-Leibler divergence and its variance, we have used the mean-variance rule and some examples have been illustrated. Further analysis of these dispersion indices could be done in order to compare distributions under different assumptions.

\section*{Acknowledgements}
Francesco Buono, Camilla Calì and Maria Longobardi are members of the research group GNAMPA of INdAM (Istituto Nazionale di Alta Matematica). Francesco Buono and Maria Longobardi are partially supported by MIUR - PRIN 2017, project ‘‘Stochastic Models for Complex Systems’’, no. 2017 JFFHSH.


\end{document}